\documentclass[12pt]{article}

\usepackage{amsmath}
\usepackage{amsfonts}
\usepackage{amsthm}
\usepackage{amssymb}
\usepackage{bm}
\usepackage[applemac]{inputenc}
\usepackage{enumerate}
\newtheoremstyle{mystyle}{}{}{\slshape}{2pt}{\scshape}{.}{ }{} 

\newtheorem{thm}{Theorem}[section]
\newtheorem{cor}[thm]{Corollary}

\newtheorem{prop}[thm]{Proposition}
\newtheorem{lemme}[thm]{Lemma}
\newtheorem{fait}[thm]{Fact}

\theoremstyle{definition}
\newtheorem{defi}[thm]{Definition}
\theoremstyle{mystyle}

\theoremstyle{remark}

\newcommand{\monster}{\mathcal U}
\newcommand{\invtypes}{\Inv_\phi(M)}
\newcommand{\fstypes}{S_{fs}(M)}

\newcommand{\finiteset}{\mathfrak P_{<\omega}(\kappa)}
\newcommand{\filter}{\mathfrak F_\kappa}

\DeclareMathOperator{\tp}{tp}

\DeclareMathOperator{\Inv}{Inv}
\DeclareMathOperator{\opp}{opp}

\def\indsym#1#2{%
 \setbox0=\hbox{$\m@th#1x$}%
 \kern\wd0%
 \hbox to 0pt{\hss$\m@th#1\mid$\hbox to 0pt{$\m@th#1^{#2}$\hss}\hss}%
 \lower.9\ht0\hbox to 0pt{\hss$\m@th#1\smile$\hss}%
 \kern\wd0}

\def\nindsym#1#2{%
 \setbox0=\hbox{$\m@th#1x$}%
 \kern\wd0%
 \hbox to 0pt{\hss$\m@th#1\not$\kern1.4\wd0\hss}
 \hbox to 0pt{\hss$\m@th#1\mid$\hbox to 0pt{$\m@th#1^{#2}$\hss}\hss}%
 \lower.9\ht0\hbox to 0pt{\hss$\m@th#1\smile$\hss}%
 \kern\wd0}

\title{Rosenthal compacta and NIP formulas}
\author{Pierre Simon\thanks{Partially supported by ValCoMo (ANR-13-BS01-0006) and by MSRI, Berkeley.}}

\date{}
\begin{document}




\maketitle

%
%
%
%
%
%

\begin{abstract}
We apply the work of Bourgain, Fremlin and Talagrand on compact subsets of the first Baire class to show new results about $\phi$-types for $\phi$ NIP. In particular, we show that if $M$ is a countable model, then an $M$-invariant $\phi$-type is Borel definable. Also the space of $M$-invariant $\phi$-types is a Rosenthal compactum, which implies a number of topological tameness properties.
\end{abstract}

Shelah introduced the independence property (IP) for first order formulas in 1971 \cite{Sh10}. Some ten years later, Poizat \cite{PoiInstable} proved that a countable theory $T$ does not have the independence property (is NIP) if and only if for any model $M$ of $T$ and type $p\in S(M)$, $p$ has at most $2^{|M|}$ coheirs (the bound a priori being $2^{2^{|M|}}$). Another way to state this result is to say that for any model $M$, the closure in $S(M)$ of a subset of size at most $\kappa$ has cardinality at most $2^{\kappa}$. Thus NIP is equivalent to a topological tameness condition on the space of types.

At about the same time, Rosenthal \cite{Rosen0} studied Banach spaces not embedding $l_1$. He showed that a separable Banach space $B$ does not contain a closed subspace isomorphic to $l_1$ if and only if the unit ball of $B$ is relatively sequentially compact in the bidual $B^{**}$, if and only if $B^{**}$ has the same cardinality as $B$. Note that an element of $B^{**}$ is by definition a function on $B^*$, the topology on $B^{**}$ is that of pointwise convergence, and $B$, identified with a subset of $B^{**}$, is dense. Shortly after this work, Rosenthal \cite{Rosen} and then Bourgain, Fremlin and Talagrand \cite{BFT} extended the ideas of this theorem and studied systematically the pointwise closure of subsets $A$ of continuous functions on a Polish space. It turns out that there is a sharp dichotomy: either the closure $\bar A$ contains non-measurable functions or all functions in the closure can be written as a pointwise limit of a \emph{sequence} of elements of $A$. In the latter case, the closure has size at most $2^{|A|}$. It turns out that this dichotomy corresponds to the NIP/IP dichotomy in an explicit way: see Fact \ref{fait_bft} (v).

The theory of compact subsets of Baire 1 functions, also known as \emph{Rosenthal compacta}, has received a lot of attention since both in general topology and set theory. See for example \cite{gode} and \cite{TodorRosen}.

The goal of this paper is to see what the Bourgain--Fremlin--Talagrand theory can tell us about NIP formulas. On the one hand, it leads us to consider new tameness properties of the space of types, whose proofs turn out to be easy with standard model-theoretic tools (Section \ref{sec_type}). On the other hand, it can be applied to prove results about invariant types for which we know no model-theoretic proof.

In particular we show the following (which is a concatenation of Propositions \ref{prop_frechetcount} and \ref{prop_isrosenthal} along with Theorem \ref{th_baireone}).

\begin{thm}\label{th_mainintro} ($T$ is countable.)
Let $\phi(x;y)$ be an NIP formula and $M$ a countable model. Let $p\in S_\phi(\monster)$ be a global $M$-invariant $\phi$-type. Then $p$ is Borel-definable: more precisely the set $\{q\in S_y(M): p\vdash \phi(x;b)$ for $b\in q(\monster)\}$ is both an $F_\sigma$ and a $G_\delta$ subset of $S_y(M)$.

The set $\invtypes$ of global $M$-invariant $\phi$-types is a Rosenthal compactum. In particular: If $Z\subseteq S_\phi(\monster)$ is a family of $M$-invariant $\phi$-types and $p$ is in $\bar Z$---the topological closure of $Z$---then $p$ is the limit of a sequence of elements of $Z$.
\end{thm}

The first point (Borel definability) was proved by Hrushovski and Pillay \cite{NIP2} assuming that the full theory is NIP. In fact their proof works if just $\phi(x;y)$ is assumed to be NIP, as long as the partial type $p$ extends to a complete $M$-invariant type. In a general theory, this need not be the case (see Section 5, Example 1 of \cite{CherKapl} for a example). The second point in the theorem is new even for NIP theories.

We will actually prove more general results which do not assume that $M$ or $T$ is countable. The proofs use two ingredients: first a theorem from \cite{InvTypes} which gives a new description of invariant $\phi$-types for an NIP formula $\phi$. Using it, we can show that the set $\invtypes$ is a Rosenthal compactum (when $M$ is countable). We could then simply apply the theory of Bourgain--Fremlin--Talagrand to obtain results such as Theorem \ref{th_mainintro}. However, to keep this paper self-contained and to remove the assumption that $M$ is countable, we will reprove everything from scratch. We want to make it clear that all the proofs in this paper (to the exception of Lemma \ref{lem_top} which is possibly new) are very closely adapted from previous works: mostly \cite{BFT} and \cite{TodorBookTopo}. The only new feature is that we do not work over a second countable space, but this poses no difficulty once the right dictionary is found (we have to replace sequences with more complicated families).

Let us say a few words about applications. The fact that invariant types in NIP theories are Borel-definable is fundamental for the theory of Keisler measures as developed in \cite{NIP2}. Using the results presented here, we can extend this theory to the case of an NIP formula $\phi(x;y)$ in an arbitrary theory. Furthermore, the fact that the closure of invariant types is witnessed by convergent sequences is used in \cite{tamedyn} to prove that, for a definably amenable group $G$, the map $p\mapsto \mu_p$ which sends an $f$-generic type to the associated $G$-invariant measure is continuous. We expect more applications to be found in the future.

Finally, we point out that Rosenthal's dichotomy was imported in dynamics through the work of K\"ohler \cite{kohler} and Glasner \cite{glasner}. From there, the relationship with NIP was noticed independently by Chernikov and myself in the work \cite{tamedyn} mentioned above and by Ibarluc\'ia \cite{Ibar} in the context of $\aleph_0$-categorical structures and automorphism groups.

\smallskip
This paper is organized as follows: In the first section, we present the relevant part of the work of Bourgain, Fremlin and Talagrand. In the second section, we state our main results and give self-contained definitions and proofs (apart from Fact \ref{fait_invkappa} which comes from one of our previous works \cite{InvTypes}). We consider first a model of arbitrary cardinality and then specialize the results to the countable case, where statements are slightly simplified by the use of sequences.

\smallskip
We would like to thank Tom\'as Ibarluc\'ia and the referee for pointing out a number of mistakes and suggesting various improvements.

\section{Rosenthal compacta}

This section surveys part of the work of Bourgain, Fremlin and Talagrand \cite{BFT} on relatively compact subsets of the first Baire class. Nothing here is needed in the rest of the paper, since we will repeat all the definitions and will not refer to it in the proofs.

\subsection{Baire 1 function}

Let $X$ be a Polish space.

\begin{defi}
A function $f:X\to \mathbb R$ is of Baire class 1 if it can be written as the pointwise limit of a sequence of continuous functions.

The set of Baire class 1 functions on $X$ is denoted by $B_1(X)$. We will always equip it with the topology of pointwise convergence, that is the topology induced from $\mathbb R^X$.
\end{defi}

The following is the well-known characterization of Baire class 1 functions due to Baire. See later Theorem \ref{th_baireone} for a proof in a slightly different framework.

\begin{fait}\label{fait_baire}
Let $f:X\to \mathbb R$. The following are equivalent:

(i) $f$ is of Baire class 1;

(ii) $f^{-1}(F)$ is a $G_\delta$ for every closed $F\subseteq \mathbb R$;

(iii) For any closed $F\subseteq X$, $f|_F$ has a point of continuity in the induced topology on $F$.
\end{fait}

\subsection{Relatively compact subsets of Baire 1 functions}

One motivation of \cite{BFT} was to answer some questions left open by Rosenthal \cite{Rosen} about the space of Baire class 1 functions on a Polish space. The authors end up proving much more general results. We will only give the particular statements relevant to us.

\begin{fait}[\cite{BFT}]\label{fait_bft}
Let $A\subseteq C(X)$ be a countable pointwise bounded family of continuous functions from $X$ to $\mathbb R$, then the following are equivalent:

(i) $A$ is relatively sequentially compact in $\mathbb R^X$ (every sequence of elements of $A$ has a subsequence which converges in $\mathbb R^X$);

(ii) $A$ is relatively compact in $B_1(X)$;

(iii) all the functions in the closure of $A$ in $\mathbb R^X$ are Borel-measurable;

(iv) the closure of $A$ in $\mathbb R^X$ has cardinality $<2^{2^{\aleph_0}}$;

(v) If $\alpha<\beta$ and $(x_n:n<\omega)$ is a sequence in $A$, then there is an $I\subseteq \omega$ such that
\[\{t:t\in X, x_n(t)\leq \alpha~  \forall n\in I, x_n(t)\geq \beta~  \forall n\in \omega \setminus I\}=\emptyset.\]
\end{fait}

The last condition is essentially the NIP property for continuous logic.

\begin{defi}\label{defi_angelic}
A regular Hausdorff space is \emph{angelic} if

(i) every relatively countably compact set is relatively compact;

(ii) the closure of a relatively compact set is precisely the set of limits of its sequences.
\end{defi}

\begin{fait}[\cite{BFT} Theorem 3F]
The space $B_1(X)$ equipped with the topology of pointwise convergence is angelic.
\end{fait}

Condition (i) in the definition of angelic was shown to hold for $B_1(X)$ by Rosenthal in \cite{Rosen}. He also made progress towards (ii).

\begin{defi}
A compact Hausdorff space $K$ is a \emph{Rosenthal compactum} if it can be embedded in the space $B_1(X)$ of functions of Baire class 1 over some Polish space $X$.
\end{defi}

A second countable (equiv. metrizable) Hausdorff compact space is a Rosenthal compactum, but the converse need not hold. Rosenthal compacta are Hausdorff compact spaces which share some nice properties with metrizable spaces even though they might not be metrizable themselves. In particular, a Rosenthal compactum is angelic.
For more on Rosenthal compacta, see for example \cite{TodorRosen} and references therein.

\section{Model theoretic results}

\subsection{Generalized Baire 1 functions}

We are interested in properties of functions from some topological space $X$ to $\{0,1\}$. To keep notations short, we will write $f^{-1}(0)$ and $f^{-1}(1)$ for the preimages of the singletons $\{0\}$ and $\{1\}$.

\begin{defi}
Let $X$ be any topological space. We define the following subspaces of the set of functions from $X$ to $\{0,1\}$.

$\cdot$ $B^{\S}(X)$: the set of functions $f:X\to \{0,1\}$ such that $\overline {f^{-1}(0)}\cap \overline {f^{-1}(1)}$ has empty interior.

$\cdot$ $B_r^{\S}(X)$: the set of functions $f:X\to \{0,1\}$ such that $f|_F \in B^{\S}(F)$ for any closed non-empty $F\subseteq X$.
\end{defi}

In all that follows, we fix an infinite cardinal $\kappa$. We will consider spaces $X$ with the following property:

\begin{description}
\item[$\boxtimes_\kappa$] $X$ is a compact Hausdorff totally disconnected space admitting a base of the topology of size at most $\kappa$.

\end{description}

Let $\finiteset$ be the set of finite subsets of $\kappa$. Let $\filter$ be the filter on $\finiteset$ generated by the sets $T_j=\{i\in \finiteset: i\supseteq j\}$ where $j$ ranges in $\finiteset$.

We will say that a family $(x_i:i\in \finiteset)$ of points in a topological space $Y$ is $\filter$-convergent to $x_*\in Y$ if for any neighborhood $U$ of $x_*$, the set $\{i\in \finiteset : x_i\in U\}$ belongs to $\filter$. We then write $x_*=\lim_{\filter} x_i$.

\begin{defi}
For $X$ satisfying $\boxtimes_\kappa$, we let $B_1(X)$ be the set of functions $f:X\to \{0,1\}$ which can be written as $\lim_{\filter} f_i$, where $(f_i:i\in \finiteset)$ is a family of continuous functions from $X$ to $\{0,1\}$.
\end{defi}

\begin{thm}\label{th_baireone}
Let $X$ satisfy $\boxtimes_\kappa$ and let $f:X\to \{0,1\}$. Then (i) and (ii) are equivalent and (iii) implies them. If $\kappa=\aleph_0$, then the three statements are equivalent.

(i) $f\in B_1(X)$;

(ii) $f^{-1}(1)$ can be written both as $\bigcup_{i<\kappa} F_i$ and as $\bigcap_{i<\kappa} G_i$, where the $F_i$'s are closed and the $G_i$'s open.

(iii) $f\in B_r^{\S}(X)$;
\end{thm}
\begin{proof}
(i) $\Rightarrow$ (ii): Write $f=\lim_{\filter} f_i$. Then for $x\in X$, $f(x)=1$ if and only if $x\in \bigcup_{i\in \finiteset} \bigcap_{j\supseteq i} f^{-1}_j(1)$. Since each $\bigcap_{j\supseteq i} f^{-1}_j(1)$ is closed and since the complement $f^{-1}(0)$ can be written in the same way, (ii) is satisfied.

(ii) $\Rightarrow$ (i): Write $f^{-1}(1)=\bigcup_{i<\kappa} F_i = \bigcap_{i<\kappa} G_i$. For $i\in \finiteset$, let $O_i\subseteq X$ be a clopen set such that $\bigcup_{k\in i} F_k \subseteq O_i$ and $O_i \subseteq \bigcap_{k\in i} G_k$. Set $f_i=1_{O_i}$. Then $f=\lim_{\filter} f_i$.

(iii) $\Rightarrow$ (ii): Assume that $f\in B_r^{\S}(X)$ and let $A=f^{-1}(1)$. It is enough to show that $A$ is the union of $\leq \kappa$ closed sets. Suppose this is not the case. Define $\mathcal U=\{U\subseteq X$ clopen $:U\cap A$ can be written as $\bigcup_{i<\kappa} F_i$, $F_i$ closed$\}$. Let $G=\bigcup \mathcal U$. Then, using $\boxtimes_\kappa$, $G\cap A$ is a union of $\kappa$ many closed sets. Let $F=X\setminus G$, a closed non-empty subset of $X$. Then $f|_F\in B^{\S}(F)$, which implies that there is some clopen $V$ such that $V\cap F\neq \emptyset$ and either $(V\cap F)\cap A=\emptyset$ or $V\cap F \subseteq A$. Then both $(V\cap F)\cap A$ and $(V\cap G)\cap A$ can be written as a union of $\kappa$ many closed sets, hence $V\in \mathcal U$. Contradiction.

(ii) $\Rightarrow$ (iii): Assume that $\kappa=\aleph_0$ and take a closed $F\subseteq X$. Write $f^{-1}(1)=\bigcup_{i<\omega} F_i=\bigcap_{i<\omega} G_i$ as in (ii). Then $F=\bigcup_{i<\omega} (F_i\cap F)\cup (G_i^c\cap F)$. Let $U\subseteq F$ be a relatively open subset. As $F$ is compact, it is Baire and for some $i$, either $F_i\cap U$ or $G_i^c\cap U$ has non-empty interior relative to $F$. This implies that $U$ cannot lie in $\overline{f^{-1}(0)} \cap \overline{f^{-1}(1)}$ and we conclude that $f|_F \in B^{\S}(F)$.
\end{proof}

\begin{lemme}\label{lem_top}
Let $\pi:L\to K$ be a continuous surjection between two compact spaces. Let $f:K\to \{0,1\}$ be a function. Assume that $\pi^*f\in B^{\S}_r(L)$, then $f\in B^{\S}_r(K)$.
\end{lemme}
\begin{proof}
Restricting the situation to a closed subset, it is enough to show that $f\in B^{\S}(K)$. Assume not, then $\overline{f^{-1}(0)}\cap \overline{f^{-1}(1)}$ contains some non-empty open set $V$. Let $K'=\overline V$. Then every element of $V$ is in the closure of both $f^{-1}(0)\cap V$ and $f^{-1}(1)\cap V$. Hence also $f\notin B^{\S}(K')$ and $K'=(\overline{f^{-1}(0)\cap K'} )\cap (\overline{f^{-1}(1)\cap K'})$. Replacing $K$ by $K'$ we may (and will) assume that $K= \overline{f^{-1}(0)}= \overline{f^{-1}(1)}$.

\smallskip \noindent
\underline{Claim}: There is a minimal closed $L'\subseteq L$ such that $\pi[L']=K$.

Proof: By Zorn's lemma, it is enough to show that if we are given a decreasing sequence $(L_i:i<\alpha)$ of closed subsets of $L$ such that $\pi[L_i]=K$, then $\pi[\cap L_i]=K$. Let $b\in K$, then the sequence $(L_i \cap \pi^{-1}(\{b\}):i<\alpha)$ is a non-increasing sequence of non-empty closed subsets of the compact set $L$. Therefore its intersection is non-empty. This proves the claim.

\smallskip
Now we may replace $L$ by $L'$ as given by the claim. Hence from now on, for any proper closed $F\subseteq L$, $\pi[F]\neq K$. As $\pi^*f \in B^{\S}(L)$, $L \neq \overline{\pi^*f^{-1}(0)} \cap \overline{\pi^*f^{-1}(1)}$. Hence at least one of $\pi^*f^{-1}(0)$ or $\pi^*f^{-1}(1)$ has non empty interior. Assume for example that there is a non-empty open set $W \subseteq \pi^*f^{-1}(0)$. Let $F=L\setminus W$. By the minimality property of $L$, $U = K \setminus \pi[F]$ is a non-empty open set. But $\pi^{-1}(U)\subseteq W \subseteq \pi^*f^{-1}(0)$. Hence $U\subseteq f^{-1}(0)$ contradicting the fact that $K= \overline{f^{-1}(1)}$.
\end{proof}

\subsection{NIP formulas and the space of types}\label{sec_type}

We let $S_x(M)$ denote the space of complete types over $M$ in the variable $x$. Also if $\Delta(x;y)$ is a formula or a set of formulas, then $S_\Delta(M)$ denotes the space of $\Delta$-types over $M$.

Recall that a formula $\phi(x;y)$ is NIP if and only if there does not exist (in the monster model $\monster$) an infinite set $A$ of $|x|$-tuples and for each $I\subseteq A$, a $|y|$-tuple $b_I$ such that $$\monster \models \phi(a;b_I) \iff a\in I,\quad \text{ for all }a\in A.$$
One obtains an equivalent definition if one exchanges the roles of $x$ and $y$. Also the definition is equivalent to saying that for any indiscernible sequence $(a_i:i<\omega)$ of $|x|$-tuple and every $|y|$-tuple $b$, there are only finitely many $i<\omega$ for which we have $\neg (\phi(a_i;b)\leftrightarrow \phi(a_{i+1};b))$. (Again, one may exchange the roles of $x$ and $y$.) See \cite{NIPbook} for details on this.

Everything in this section already appeared in \cite{InvTypes}, but we recall it here for convenience. The following fact is well-known (at least when the full theory is NIP, but the proof is the same in the local case).

\begin{fait}\label{fact_morleyloc}
Assume that the formula $\phi(x;y)$ is NIP. Let $p,q\in S_x(\monster)$ be $A$-invariant types and we let $p_\phi,q_\phi$ denote the restrictions of $p$ and $q$ res\-pectively to instances of $\phi(x;y)$ and $\neg \phi(x;y)$. If $p^{(\omega)}|_A=q^{(\omega)}|_A$, then $p_\phi=q_\phi$.
\end{fait}
\begin{proof}
Assume that for example $p\vdash \phi(x;b)$ and $q\vdash \neg \phi(x;b)$ for some $b\in \monster$. Build inductively a sequence $(a_i:i<\omega)$ such that:

$\cdot$ when $i$ is even, $a_i \models p \upharpoonright Aba_{<i}$;

$\cdot$ when $i$ is odd, $a_i\models q\upharpoonright Aba_{<i}$.

Then by hypothesis, the sequence $(a_i:i<\omega)$ is indiscernible (its type over $A$ is $p^{(\omega)}|_A=q^{(\omega)}|_A$) and the formula $\phi(x;b)$ alternates infinitely often on it, contradicting NIP.
\end{proof}

The following proposition and proof come from \cite[Lemma 2.8]{InvTypes}. It is inspired by point (ii) in the definition of angelic (Definition \ref{defi_angelic}).

\begin{prop}\label{prop_frechetkappa}
Let $A$ be a set of parameters of size $\kappa$ and let $\Delta=\{\phi_i(x;y_i)\}$ be a set of NIP formulas of size $\leq \kappa$. Let $q$ be a global $\Delta$-type finitely satisfiable in $A$. Then there is a family $(b_i:i\in \finiteset)$ of points in $A$ such that $\lim_{\filter}(\tp_{\Delta}(b_i/\monster))=q$.
\end{prop}
\begin{proof}
Taking a reduct if necessary, we may assume that the language has size at most $\kappa$. Extend $q$ to some complete type $\tilde q$ finitely satisfiable in $A$ and let $I=(b'_i:i<\omega)$ be a Morley sequence of $\tilde q$ over $A$. List the formulas in $\tilde q|_{AI}$ as $(\phi_k(x;c_k):k<\kappa)$. For $i\in \finiteset$, take $b_i\in A$ realizing $\bigwedge_{k\in i} \phi_k(x;c_k)$. Assume that the family $(\tp_{\Delta}(b_i/\monster):i\in \finiteset)$ does not converge to $q$ along $\filter$ and let $\phi(x;c)\in q$ witness it. Let $\mathcal D$ be an ultrafilter on $\finiteset$ extending $\filter$ and containing $\{i\in \finiteset:\models \neg \phi(b_i;c)\}$. Let $\tilde q'=\lim_{\mathcal D}(\tp(b_i/\monster))$. Then as $\mathcal D$ contains $\filter$, we have $\tilde q'|_{AI}=\tilde q|_{AI}$. By an easy induction, this implies $\tilde q'^{(\omega)}|_A=\tilde q^{(\omega)}|_A$. By Fact \ref{fact_morleyloc}, $\tilde q$ and $\tilde q'$ agree on $\Delta$-formulas, but this is a contradiction since $\tilde q'\vdash \neg \phi(x;c)$.
\end{proof}

\begin{cor}
Let $\Delta$ be as in the previous proposition and let $A\subseteq S_\Delta(\monster)$ be a set of $\Delta$-types of size at most $\kappa$. Let $q\in S_\Delta(\monster)$ be in the topological closure of $A$. Then there is a family $(p_i:i\in \finiteset)$ of elements of $A$ such that $\lim_{\filter}p_i=q$.
\end{cor}
\begin{proof}
Realize each type $p\in A$ by an element $a_p$ in some larger monster model $\monster_1$. Then $q$ extends to a type $\tilde q$ over $\monster_1$ which is finitely satisfiable in $\{a_p :p\in A\}$. By the previous proposition, there is a family $(a_{p_i}:i\in \finiteset)$ such that $\lim_{\filter}(\tp_{\Delta}(a_{p_i}/\monster'))=\tilde q$. Restricting to $\monster$, we see that the family $(p_i :i\in \finiteset)$ converges to $q$.
\end{proof}

\subsection{Invariant $\phi$-types}

Let $T$ be any theory, $M\models T$ and $\phi(x;y)$ an NIP formula. Assume that both $M$ and $T$ have size at most $\kappa$. In what follows, $\phi^1$ means $\phi$ and $\phi^0$ means $\neg \phi$.

Let $\Inv_\phi(M)\subset S_{\phi}(\monster)$ denote the space of global $M$-invariant $\phi$-types. Given $p\in \Inv_\phi(M)$, define the function $d_p : S_y(M)\to \{0,1\}$ by $d_p(q)=1$ if $p\vdash \phi(x;b)$, for some/any $b\in q(\monster)$ and $d_p(q)=0$ otherwise.

Note that $S_y(M)$ satisfies property $\boxtimes_\kappa$ above.
Our goal now is to show that $d_p\in B_r^{\S}(S_y(M))$.

Let $S_{fs}(M)\subset S_y(\monster)$ be the space of global types in the variable $y$ finitely satisfiable in $M$ and let $S_{fs}^{\phi^{\opp}}(M)$ be the space of global $M$-finitely satisfiable $\phi^{\opp}$-types (where $\phi^{\opp}(y;x)=\phi(x;y)$). We have two natural projection maps:

$\pi: S_{fs}(M) \to S_y(M)$, which assigns to a type its restriction to $M$, and

$\pi_0: S_{fs}(M) \to S^{\phi^{\opp}}_{fs}(M)$, which sends a type to its reduct to instances of $\phi^{\opp}$.

Given $p\in \invtypes$, define $f_p:\fstypes \to \{0,1\}$ as $\pi^*(d_p)$. Also given $s\in S_x(M)$, let $a\models s$ and define the map $\hat s:\fstypes \to \{0,1\}$ by $\hat s(q)=1$ if $q\vdash \phi(a;y)$ and $\hat s(q)=0$ otherwise. Note that this map factors through $S^{\phi^{\opp}}_{fs}(M)$ and is a continuous function on $\fstypes$.

The following is shown in \cite{InvTypes}, Proposition 2.11. The moreover part is Proposition 2.13 there.

\begin{fait}\label{fait_invkappa}
The map $f_p$ factors through $S^{\phi^{\opp}}_{fs}(M)$ and moreover 

$f_p\in \overline{\{\hat s : s\in S_x(M)\}}$.
\end{fait}

\begin{prop}\label{prop_bftreprovedkappa}
The function $f_p$ is in $B^{\S}_r(\fstypes)$.
\end{prop}
\begin{proof}
Let $X\subseteq \fstypes$ be a non-empty closed set. From now on all topological notions are meant relative to $X$. If $f_p\notin B^{\S}(X)$, then restricting $X$ further as in the proof of Lemma \ref{lem_top}, we may assume that $\overline {f_p^{-1}(0)}=\overline{f_p^{-1}(1)}=X$. To any finite sequence $\bar V=(V_1,\ldots,V_n)$ of non-empty open subsets of $X$, we associate a type $s_{\bar V}\in S_x(M)$ as follows. For each $i$, $f_p$ is not constant on $V_i$ by hypothesis, hence there is a pair $(a_i,b_i)$ of points in $V_i$ such that $f_p(a_i)=1$ and $f_p(b_i)=0$. Having chosen such a pair for each $i$, we can apply Fact \ref{fait_invkappa} to find some $s_{\bar V}\in S_x(M)$ such that $\hat s_{\bar V}(a_i)=1$ and $\hat s_{\bar V}(b_i)=0$.

Now we construct a sequence $(s_l)_{l<\omega}$ of types in $S_x(M)$ inductively. Start with $\bar V_0=(X)$ and define $s_0=s_{\bar V_0}$ as above. Then set $\bar V_1=(U_0,U_1)$, where $U_0=\hat {s}_0^{-1}(0)$ and $U_1=\hat {s}_0^{-1}(1)$. Define $s_1=s_{\bar V_1}$. Having defined $s_k$, $k<l$, build a family of open sets $\bar V_l=(U_\eta:\eta \in \{0,1\}^{l})$ where $U_\eta=\bigcap_{k<l} \hat{s}_k^{-1}(\eta(k))$. The construction ensures that each $U_\eta$ is non-empty. Let $s_l=s_{\bar V_l}$.

Having done this for all $l<\omega$, take realizations $a_l$ of the types $s_l$, $l<\omega$. Then by construction, for any function $\eta: \omega \to \{0,1\}$, the type 

$\{\phi(a_l;y)^{\eta(l)}:l<\omega\}$

\noindent
is consistent, contradicting NIP.
\end{proof}

We are now ready to prove our main theorem.

\begin{thm}\label{th_mainth}
Let $p\in \invtypes$, then the function $d_p$ is in $B_r^{\S}(S_y(M))$.
\end{thm}
\begin{proof}
By Proposition \ref{prop_bftreprovedkappa}, the function $f_p$ is in $B^{\S}_r(\fstypes)$. Using Lemma \ref{lem_top} with $L=\fstypes$ and $K=S_y(M)$, we conclude that $d_p$ is in $B^{\S}_r(S_y(M))$.
\end{proof}

We draw some consequences of this result. The following proof comes from \cite[Chapter 10, Corollary 4]{TodorBookTopo}.

\begin{lemme}
Let $Z\subseteq \invtypes$ be any subset and assume that $p\in \overline Z$. Then there is a subset $Z_0\subseteq Z$ of size at most $\kappa$ such that $p\in \overline Z_0$.
\end{lemme}
\begin{proof}
For any $A\subseteq Z$, and $n<\omega$, define $A^{(n)}$ to be the set of tuples $\bar s\in S_y(M)^n$ for which there is no $q\in A$ such that $d_q$ agrees with $d_p$ on $\bar s$. We are looking for a subset $Z_0\subseteq Z$ of size $\leq \kappa$ such that $Z_0^{(n)}= \emptyset$ for all $n$.

Set $A_0=\emptyset$ and build by induction on $\alpha$ a sequence $A_0 \subseteq \cdots \subseteq A_\alpha \subseteq \cdots$ of subsets of $Z$ of size at most $\kappa$ such that for each $\alpha$, for some $n=n(\alpha)$, $\overline {A_\alpha^{(n)}} \supsetneq \overline {A_{{\alpha+1}}^{(n)}}$. This process must stop at some ordinal $\alpha<\kappa^+$ because each $S_y(M)^n$ has a base of open sets of size $\kappa$. We then have that for any $B\supseteq A_\alpha$ of size at most $\kappa$, for all $n<\omega$, $\overline {B^{(n)}} = \overline {A_\alpha^{(n)}}$.

Now set $A=A_\alpha$. If for all $n$, $A^{(n)}=\emptyset$, then we are done. Otherwise, fix some $n$ for which $A^{(n)}$ is not empty and let $K=\overline {A^{(n)}}$. Fix $S\subseteq A^{(n)}$ a dense subset of size $\leq \kappa$ and enumerate it as $S=(x_k:k<\kappa)$. For $q\in Z$, let $\delta_q: K\to \{0,1\}$ be defined by $\delta_q(\bar s)=0$ if $d_q$ agrees with $d_p$ on $\bar s$, and $\delta_q(\bar s)=1$ otherwise.

We claim that $\delta_q$ is in $B^{\S}(K)$: For $i,j\in \{0,1\}$ define $W_{i,j}\subseteq S_y(M)$ as the intersections of the interiors of $d_p^{-1}(i)$ and of $d_q^{-1}(j)$. As $d_p$ and $d_q$ are in $S^{\S}(S_y(M))$ by Theorem \ref{th_mainth}, $W=\bigcup_{i,j} W_{i,j}$ is a dense open set. Hence $W^n$ is a dense open set in $S_y(M)^n$. But $\delta_q$ is locally constant on $W^n$: $W^n$ is disjoint from $\overline{\delta_q^{-1}(0)}\cap \overline{\delta_q^{-1}(1)}$. We conclude that $\delta_q$ is in $B^{\S}(K)$ as claimed.

For $i\in \finiteset$, find some $q_i\in Z$ such that $\delta_{q_i}$ is equal to 0 on $\{x_k:k\in i\}$ (exists as $p\in \overline Z$). Then for all $x\in S$, $\lim_{\filter} \delta_{q_i}(x)=0$. Let $B=\{q_i\}_{i\in \finiteset} \cup A$. Take $\mathcal D$ an ultrafilter extending $\filter$ and set $q =\lim_{\mathcal D}q_i$. Then $\delta_q$ is equal to 0 on $S$ and to 1 on $B^{(n)}$. Both $S$ and $B^{(n)}$ are dense subsets of $K$. This contradicts the fact that $\delta_q \in B^{\S}(K)$.
\end{proof}

\begin{prop}\label{prop_frechetinv}
Let $Z\subseteq \invtypes$ be any subset and assume that $p\in \overline Z$. Then there is a family $(q_i:i\in \finiteset)$ of elements of $Z$ with $p=\lim_{\filter} q_i$.
\end{prop}
\begin{proof}
By the previous lemma, we may assume that $Z=\{q_i:i<\kappa\}$ has size at most $\kappa$ (allowing repetitions in the $q_i$'s if $|Z|<\kappa$). Fix some model $N$ containing $M$ and $|M|^+$-saturated. For each $i<\kappa$, let $a_i \models q_i|N$. Let $\mathcal D$ be an ultrafilter on $\kappa$ such that $\lim_{\mathcal D} q_i = p$. Define $\tilde p= \lim_{\mathcal D} \tp_\phi(a_i/\monster)$. This type is finitely satisfiable in $\{a_i:i<\kappa\}$. By Proposition \ref{prop_frechetkappa}, there is a subfamily $(a_{\eta(i)}:i\in \finiteset)$ such that $\lim_{\filter} \tp_\phi(a_{\eta(i)}/\monster)=\tilde p$. Then $p|_N=\lim_{\filter} q_{\eta(i)}|_N$. All types involved are $M$-invariant, hence they are determined by their restriction to $N$ and we conclude that $p=\lim_{\filter} q_{\eta(i)}$.
\end{proof}

\subsection{The case $\kappa=\aleph_0$}\label{sec_countable}

Assume in this section that $\kappa=\aleph_0$ and hence $T$ is countable. The results above are slightly simpler to state in this case, because we can replace $\mathfrak F_{\aleph_0}$-convergent families by convergent sequences. In fact the two notions are essentially equivalent: given a $\mathfrak F_{\aleph_0}$-convergent family $(f_i:i\in \mathfrak P_{<\omega}(\aleph_0))$, the sequence $(f_n:n<\omega)$ is convergent where $n$ is identified here with $\{0,\ldots,n-1\}$. Conversely, if $(f_n:n<\omega)$ is any sequence, then it converges if and only if the family $(f_i:i\in \mathfrak P_{<\omega}(\aleph_0))$ is $\filter$-convergent, where $f_i=f_n$ for $n$ maximal such that $n\subseteq i$.

So Proposition \ref{prop_frechetkappa} becomes the following (which already appeared in \cite{InvTypes}).

\begin{lemme}\label{lem_frechetcount}
Let $A$ be countable and $\Delta=\{\phi_i(x;y_i)\}$ a countable set of NIP formulas. Let $q$ be a global $\Delta$-type finitely satisfiable in $A$. Then there is a convergent sequence $(b_i:i<\omega)$ of points in $A$ such that $\lim(\tp_{\Delta}(b_i/\monster))=q$.
\end{lemme}

\begin{cor}
Let $\Delta$ be as above. Then the space $S_\Delta(\monster)$ of $\Delta$-types over $\monster$ is sequentially compact.
\end{cor}
\begin{proof}
Let $(p_i:i<\omega)$ be a sequence of $\Delta$-types over $\monster$, which we may assume to be pairwise distinct. In a bigger monster model $\monster_1$, realize each $p_i$ by a point $a_i$ and let $A= \{a_i:i<\omega\}$. Let $Z = \{\tp_\Delta(a_i/\monster_1):i<\omega\}$ and let $q$ be an accumulation point of $Z$. Then by the previous lemma, $q$ is the limit of a subsequence $(\tp_\Delta(a_{\eta(i)}/\monster_1):i<\omega)$. In particular, the subsequence $(p_{\eta(i)}:i<\omega)$ converges in $S_\Delta(\monster)$.
\end{proof}

When $\kappa=\aleph_0$, Theorem \ref{th_baireone} boils down to the usual characterization of Baire class 1 functions as recalled in Fact \ref{fait_baire}. Note that $S_y(M)$ is now a Polish space. Using the notations of the previous section, we deduce from Theorem \ref{th_mainth} that if $M$ is countable, $\phi(x;y)$ is NIP and $p\in Inv_\phi(M)$, then the the function $d_p: S_y(M) \to \{0,1\}$ is of Baire class 1 in the usual sense.

Finally Proposition \ref{prop_frechetinv} becomes the following statement.

\begin{prop}\label{prop_frechetcount}
Let $T$ and $M$ be countable, $\phi(x;y)$ NIP and $Z\subseteq \invtypes$ be any subset. Let $p\in \overline Z$. Then there is a sequence $(q_n:n<\omega)$ of elements of $Z$ converging to $p$.
\end{prop}

In fact, since $S(M)$ is a Polish space, we have the more precise result.

\begin{prop}\label{prop_isrosenthal}
Let $T$ and $M$ be countable, $\phi(x;y)$ NIP, then the set $\invtypes$ is a Rosenthal compactum.
\end{prop}
\begin{proof}
The set $\invtypes$ can be identified with a closed subspace of functions from $S_y(M)$ to $\{0,1\}$ and by Theorem \ref{th_mainth}, those functions are all of Baire class 1.
\end{proof}

We conclude with a theorem of Bourgain, which he stated for Rosenthal compacta in \cite{BourGdelta}. The proof we give is his, adapted to our context.

By a $G_\delta$ point $x$ of a space $F$, we mean a point $x$ such that the singleton $\{x\}$ can be written as an intersection of at most countably many open subsets of $F$.

\begin{prop}
Assume that $T$ and $M$ are countable and $\phi(x;y)$ is NIP. Then any closed subset $F\subseteq \invtypes$ contains a dense set of $G_\delta$ points.
\end{prop}
\begin{proof}
For $X\subseteq S_y(M)$ and $\epsilon \in\{0,1\}$, let $C_\epsilon(X)$ be the set of types $p\in \invtypes$ such that $d_p$ restricted to $X$ is constant equal to $\epsilon$. It is a closed subset of $\invtypes$.

Fix a closed subset $F\subseteq \invtypes$, set $F_0 = F$, $U_0=\emptyset$ and we try to build by induction on $\alpha<\aleph_1$:

$\cdot$ a sequence of non-empty closed sets $F_\alpha\subseteq \invtypes$ such that $F_{\alpha}=W_\alpha\cap  \bigcap_{\beta<\alpha} F_\beta$, where $W_\alpha \subseteq \invtypes$ is clopen (in particular $F_\alpha \subseteq F_\beta$ for $\alpha\geq \beta$);

$\cdot$ an increasing sequence of open sets $U_\alpha \subseteq S_y(M)$ such that all the functions $d_p$ for $p\in F_\alpha$ agree on $U_\alpha$.\\
Since there is no increasing sequence of open subsets of $S_y(M)$ of length $\aleph_1$, this construction must stop at some $\alpha_*<\aleph_1$. Let then $F_* = \bigcap_{\alpha < \alpha_*} F_{\alpha}= F \cap \bigcap_{\alpha<\alpha_*} W_\alpha$ and $U_* = \bigcup_{\alpha<\alpha_*} U_\alpha$. Note that $F_*$ is a closed non-empty $G_\delta$ subset of $F$ and all the functions $d_p$, $p\in F_*$, agree on $U_*$. We show that $U_* = S_y(M)$, and thus $F_*$ must be a singleton which gives what we want.

Assume not and let $K= S_y(M)\setminus U_*$. Then for $p\in F_*$, as $d_p\in B^{\S}(K)$, there is some non-empty clopen set $V\subseteq K$ such that $d_p$ is constant on $V$. Thus $p \in C_\epsilon(V)$ for some $\epsilon$. As there are countably many clopen sets in $K$, by the Baire property, there is a non-empty clopen $V\subseteq K$ such that $C_\epsilon(V) \cap F_*$ has non-empty interior in $F_*$. Fix such a $V$. Then we can set $U_{\alpha_*} = U_*\cup V$, and find some clopen $W_{\alpha_*}\subseteq \invtypes$ such that $F_{\alpha_*}:= F_* \cap W_{\alpha_*}$ is non-empty and included in $C_\epsilon(V)$. This contradicts maximality of the construction.
\end{proof}


\bibliography{BFT}

\end{document}